\documentclass[12pt]{elsarticle}
\usepackage{amsfonts}
\usepackage{epstopdf}
\usepackage[dvipsnames]{color}
\usepackage{amssymb}
\usepackage{verbatim}
\usepackage[mathscr]{eucal}
\usepackage{psfrag}
\usepackage{amsmath,amssymb,subfigure,graphicx}
\usepackage{amsthm}
\usepackage{cases}
\usepackage[OT1]{fontenc}
\usepackage{lineno}
\usepackage{url}
\usepackage{setspace}
\usepackage{multirow,booktabs,rotating}

\modulolinenumbers[5]
\theoremstyle{plain}
\newtheorem{theorem}{Theorem}[section]

\newtheorem{lemma}[theorem]{Lemma}
\newtheorem{proposition}[theorem]{Proposition}
\theoremstyle{property}
\newtheorem{property}{Property}
\theoremstyle{definition}
\newtheorem{definition}[theorem]{Definition}
\theoremstyle{remark}
\newtheorem{remark}[theorem]{Remark}
\numberwithin{equation}{section}
\begin{document}

\begin{frontmatter}
\title{ \bf Theta Functions and Adiabatic Curvature on a Torus}
\date{}
\author[Ching-Hao Chang]{Ching-Hao Chang}
\author[Jih-Hsin Cheng]{Jih-Hsin Cheng}
\author[I-Hsun Tsai]{I-Hsun Tsai}

\address[Ching-Hao Chang]{Department of Mathematics, Xiamen University Malaysia, Jalan Sunsuria, Bandar Sunsuria, 43900 Sepang, Selangor Darul Ehsan, Malaysia \\ chinghao.chang@xmu.edu.my}

\address[Jih-Hsin Cheng]{Institute of Mathematics, Academia Sinica and NCTS, 6F, Astronomy-Mathematics Building, No.1, Sec.4, Roosevelt Road,Taipei 10617, Taiwan \\ 
cheng@math.sinica.edu.tw}

\address[I-Hsun Tsai]{Department of Mathematics, National Taiwan University, Taipei 10617, Taiwan \\ 
ihtsai@math.ntu.edu.tw}

\date{\today}

\begin{keyword} Theta functions, Complex torus, Picard variety, Poincar\'e line bundle,
Connection, Curvature, Characters, Holonomy, Fourier-Mukai transform.
 \\
\emph{2010 MSC:} \vspace{0.2cm} Primary 32G05; Secondary 32S45, 18G40, 32C35, 32C25.
\end{keyword}



\begin{abstract} Let $M$ be a complex torus, $L_{\hat\mu}\to M$ be positive line bundles parametrized by
$\hat \mu\in {\rm Pic}^0(M)$, and $E\to {\rm Pic}^0(M)$ be a vector bundle with
$E|_{\hat\mu}\cong H^0(M, L_{\hat \mu})$.  We endow the total family $\{L_{\hat\mu}\}_{\hat\mu}$ with
a Hermitian metric that induces the $L^2$-metric on $H^0(M, L_{\hat \mu})$ hence on $E$.
By using theta functions $\{\theta_m\}_{m}$ on $M\times M$ as a family of functions on the first factor $M$  
with parameters in the second factor $M$, our computation of the full curvature tensor $\Theta_E$ of $E$ with respect to
this $L^2$-metric shows that $\Theta_E$ is essentially an identity matrix multiplied by a constant
$2$-form, which yields in particular the adiabatic curvature $c_1(E)$.
After a natural base change $M\to \hat M$ so that $E\times_{\hat M} M:=E'$,
we also obtain that $E'$ splits holomorphically into a direct sum of line bundles each of which is
isomorphic to $L_{\hat\mu=0}^*$.   Physically, the spaces $H^0(M, L_{\hat \mu})$ correspond to the
lowest eigenvalue with respect to certain family of Hamiltonian operators on $M$ parametrized by
$\hat\mu$ or in physical notation, by wave vectors $\bf k$.
\end{abstract}
\end{frontmatter}

\section{Introduction}

Let $M$ be a complex torus. To consider the set of all positive line bundles 
$L\to M$ with the same first Chern classes, one may first pick any positive
line bundle $L_0\to M$ with the required $c_1(L_0)=[\omega]$ for some closed 
$(1, 1)$ form $\omega$ which is integral, positive and of constant
coefficients. Write $\delta$ for the degree of $L_0$. For any holomorphic
automorphism $T:M\to M$, $c_1(T^*L_0)=[\omega]$ and it is well-known that
all line bundles on $M$ with the same $c_1$ can arise in this way. In fact,
it is known that $T$ is a translation $T_{\mu}:M\to M$ on $M$ for some fixed 
$\mu\in M$. We denote $T^*L_0$ by $L_{\mu}$.

This can be placed in another context by means of Poincar\'e line bundle $%
\mathrm{P}: M\times \hat M$ where $\hat M=\mathrm{Pic}^0(M)$. Let $\pi_1$, $%
\pi_2$ be the two projections of $M\times \hat M$ to $M$, $\hat M$
respectively. Write $\tilde E=\pi_1^*L_0\otimes \mathrm{P}$. Thinking of $%
{\tilde E}_{\mid_{M\times\{\hat \mu\}}}$ on $M\times\{\hat \mu\}$ as a family of
line bundles $L_{\hat\mu}$ on $M\cong M\times\{\hat\mu\}$, one has the
associated family of vector spaces $H^0(M, L_{\hat\mu})$ varying with $%
\hat\mu$. It forms a holomorphic vector bundle $E$ on $\hat M$. Similarly,
we have a holomorphic vector bundle $E^{\prime }$ on $M$ with $E^{\prime
}_{\mid_{\mu}}=H^0(M, L_{\mu})$. This type of construction is closely related to
the Fourier-Mukai transform. See \cite{P2}. There is a map $\varphi_{L_0}:M\to
\hat M$ sending $\mu\in M$ to $T_{\mu}^*L_0\otimes L_0^*\in \hat M$. For
precise notations and details, we refer to later appropriate sections.

A natural question of interest in this paper is to ask for the full
curvature of $E$. We have:

\begin{theorem}
\label{mainTheo} {\rm ($=$ Theorem $\ref{theorem 8}$.)} In the notations
as above, there exists a Hermitian metric $h_{\widetilde{E}}$ on $\widetilde{%
E}$ such that the induced $L^2$-metric on $E$, denoted by $h_E$, has the
curvature 
\begin{equation*}
\Theta(E, h_E)=(2\pi i)\omega (Id)_{\delta\times\delta}
\end{equation*}
where $(Id)_{\delta\times\delta}$ denotes the $\delta\times \delta$ identity
matrix. Therefore $c_1(E, h_E)=-\delta\omega$ (at the level of differential
forms).
\end{theorem}

Our study into this question was influenced by a related work of C. T.
Prieto \cite{P2} where he studied similar questions on compact Riemann
surfaces but restricted to $c_{1}$. Among other things, he placed his
computations in the framework of local family index theorems, and derived
the $c_{1}$ from the theorem of Bismut-Gillet-Soul\'{e} \cite{BGS} in this
regard. To invoke these theorems, the Quillen metric need be introduced as 
an extra ingredient. By contrast, we use theta functions for explicit
computations and achieve the full curvature $\Theta $ of $E$.

In fact, the above $\Theta $ is obtained via the following result of
independent interest, which appears to be of algebraic geometry in nature.

\begin{theorem}
\label{mainTheo2} {\rm (See $(\ref{Esplit1})$.)} We have $%
\varphi_{L_0}^*E=E^{\prime }$ on $M$. Moreover, 
$E^{\prime }$ splits holomorphically into a direct sum of holomorphic line
bundles each of which is isomorphic to $L_0^*$, the dual of $L_0$.
\end{theorem}

There are rich connections between these problems and physics, for which we
refer mathematically minded readers to the nice presentation by Prieto in 
\cite[Introduction]{P2}, including the term "adiabatic curvature". For
physical interest, it is desirable to compute the adiabatic curvature of 
\textit{spectral bundles} (cf. \cite{ABGM}), where our space of holomorphic
sections corresponds to the lowest eigenvalue under suitable interpretation.
Some interesting results in this direction (for higher eigenvalues) have
been obtained by Prieto in \cite{P1} and \cite{P2}. Put in this perspective,
our present work is far from being complete. Another immediate question is
to ask for the higher dimensional generalization of Theorem \ref{mainTheo}
say, on an Abelian variety. Further, our present approach is transcendental
in nature, and from the purely algebraic point of view, it is not altogether
clear how Theorem \ref{mainTheo2} can be proved in an algebraic manner. A
third question of interest appears to be a study into all of these problems
under deformation of complex structures on $M$. We hope to come back to
(some of) these questions in future publications.

We remark that the theoretical and experimental aspects of the role played
by the first Chern class $c_{1}$ have long been noticed by physicists under
study of, among others, "geometric phases in quantum systems" in general and
the quantum Hall effect in particular (cf. \cite{BM}, \cite{L}, \cite{R}).
In these settings the adiabatic curvature usually refers to the $c_{1}$ (or $%
\frac{2\pi }{i}c_{1}$) of spectral bundles associated with certain
Hamiltonian operators depending on parameters such as wave vectors (cf. \cite%
[(13.26) in p. 314]{BM}). While the theoretical/abstract formula for the
(full) curvature is already available, some physical approaches to the
actual computation are carried out using, for instance, "magnetic
translation operators" (cf. \cite{Av} and references therein) and even
noncommutative geometry methods (cf. \cite{BES}). To the best of our
understanding, these studies and explicit results focus only on $c_{1}$
rather than the full curvature tensor as done here.

The full curvature in related contexts has been of interest in the mathematical
literature. Indeed, it appears in disguise of the Chern character of the index
bundle (see \cite{Bis}) and more recently, it also plays an important role
in the work of B. Berndtsson for vector bundles associated to holomorphic
fibrations (see \cite{Ber}).
 
To outline our approach, some difficulties are in order. It is natural to
consider metrics $h_{\hat\mu}$ on $L_{\hat\mu}$ for $\mu\in\hat M$ which are
of constant curvature $\frac{2\pi}{i}\omega$. As this curvature condition
determines $h_{\hat\mu}$ only up to multiplicative constants, one is
required not only to make a choice but also, more importantly, to do it in a
consistent manner with respect to $\hat\mu$ globally. By this, among others,
we are led to the Poincar\'e line bundle $\mathrm{P\to M\times\hat M}$. But
we found it much less illuminating if we fell into the description of $%
\mathrm{P}$ in terms of complex algebraic geometry as usually given in the
literature. Fortunately, the needed differential geometric aspects on the
Poincar\'e line bundle $\mathrm{P}$ have been developed in part by \cite{DK}
from the gauge theory perspective (cf. Section $\ref{section 5}$). This is
precisely what we resort to here, and by proving an identification theorem,
we can endow $\mathrm{P}$ with certain metric geometry data (cf. Section $%
\ref{section 6}$).

Next, from the physical point of view it is natural to use the $L^2$-metric
of the system for the curvature computation. For this purpose, the explicit
theta functions as global sections are expected to deserve a try. However,
as far as the Theorem \ref{mainTheo2} is concerned, our difficulty lies in
that the choice of these functions \textit{a priori} depends on $\hat\mu$
although the curvature computation only makes use of a \textit{local} basis
of theta functions valid around $\hat\mu$, for $L_{\hat\mu}\to
M\times\{\hat\mu\}$. We are therefore led to exploit a \textit{global}
property of these ($\hat\mu$-dependent) theta functions (cf. Section $\ref%
{section 1-2}$ and Section $\ref{section 2}$). For the formulation it turns
out to get most simplified if we shift the viewpoint about parameters from $%
\hat\mu\in \hat M$ to $\mu\in M$ via the map $\varphi_{L_0}:M\to\hat M$ as
given precedingly (cf. Section $\ref{section 3}$). We thus form the theta
functions on $M\times M$ as a family of functions defined on the first $M$
as well as parametrized by the second $M$ (cf. Section $\ref{section 4}$).
In this way we can eventually accomplish a holomorphic splitting of the
vector bundle in the sense of Theorem \ref{mainTheo2}.

In retrospect, it remains somewhat unexpected why the $L^{2}$-metric
property of these global theta functions so formed, behave nicely to suit
our (computational) need. Indeed, it is only after the explicit computation
that we find this neat fact. See the main technical Lemma \ref{Lemma 4} for
details. Nevertheless, we are prompted to perceive Theorem \ref{mainTheo2}
as a conceptual picture in support of the computational result Theorem \ref%
{mainTheo} (cf. Remark $\ref{rem3r}$ and ii) of Remark $\ref{rem85}$).

\section*{Acknowledgements}

The first author is supported by Xiamen University Malaysia Research Fund
(Grant No : XMUMRF/2019-C3/IMAT/0010). The second author is supported by the
project MOST 107-2115-M-001-011 of Ministry of Science and Technology and
NCTS of Taiwan. The third author warmly thanks the Academia Sinica for
excellent working conditions that make this joint work more enjoyable.

\section{Holomorphic line bundles over the compact Riemann surface $M = V /
\Lambda$}

\label{section 1-2} The principal aim of this section is to collect the
background materials and to fix the notations for later use. Basic
references are, for instance, \cite{GH} and \cite{M}. Let $V$ be a complex
vector space of dimension 1 and $\Lambda = \mathbb{Z}\, \{ \lambda_{1},\
\lambda_{2}\} \subseteq V $ be a discrete lattice where $Im \frac{\lambda_{2}%
}{\lambda_{1}} > 0$. The compact Riemann surface $M= V / \Lambda$ is a
complex torus. Let $L$ be a holomorphic line bundle over $M$. The first
Chern class $c_{1}(L)$ of $L$ is a complete invariant of $L$ as a $%
C^{\infty} $ line bundle. The Picard group $\mathrm{Pic}\, (M)$ are the
isomorphic classes of holomorphic line bundles over $M$. The connected
component $\mathrm{Pic}^{0}\, (M)$ of $\mathrm{Pic}\, (M)$ represents all
the equivalent classes of degree 0 holomorphic line bundles over $M$.

We let $\{ d{x_{1}}$, $dx_{2} \}$ be the 1-forms on $V$ dual to $\{
\lambda_{1}, \lambda_{2} \}$, that is, $\int_{\lambda_{i}} dx_{j} =
\delta_{ij}$. In terms of this basis, any positive holomorphic line bundle $%
L $ over $M$ has a Hodge form $\omega = \delta \ dx_{1} \wedge dx_{2}$ on $M$
satisfying $c_{1}(L)= [\omega]$, $\delta \in \mathbb{N}$.

To fix the complex coordinates, choose a $\delta \in \mathbb{N}$ and let $%
e_{1}= \frac{\lambda_{1}}{\delta}$, $\tau = \tau_{1}+ i \tau_{2} =\frac{%
\lambda_{2}}{e_{1}}$. We write $\lambda_{1} = \delta e_{1}$, $\lambda_{2} =
\tau e_{1} $ with $\tau_{2}>0$. Let $z=z_{1}+iz_{2}$ with $z_{1}, z_{2} \in 
\mathbb{R}$, be the complex coordinate on $V$ (and on $M$) such that $dz$ is
dual to $e_{1}$.

We denote $z\,e_{1}\in V$ by $z$ whenever there is no danger of confusion.
One has%
\begin{equation}
\begin{cases}
dz=\delta dx_{1}+\tau dx_{2} \\ 
d\overline{z}=\delta \,dx_{1}+\overline{\tau }\,dx_{2}.%
\end{cases}%
\hspace*{170pt}  \label{2.0}
\end{equation}

We define $L_{0}$ to be the holomorphic line bundle over $M$ given by
multipliers 
\begin{equation}  \label{eqnL0}
\begin{cases}
e_{\lambda_{1}}(z) \equiv 1 \\ 
e_{\lambda_{2}}(z) = e^{- 2 \pi i z - \pi i \tau}.%
\end{cases}
\hspace*{170pt}
\end{equation}
Notice that any system of multipliers $\{\, e_{\lambda} \in \mathcal{O}%
^{*}(V)\, \}_{\lambda \in \Lambda}$ for a holomorphic line bundle $L$ on $%
M=V / \Lambda$ has to satisfy the \textbf{compatibility relations :} \newline
\begin{equation}  \label{compatibility}
e_{\lambda^{^{\prime }}}(z+\lambda)\,
e_{\lambda}(z)=e_{\lambda}(z+\lambda^{^{\prime }})\, e_{\lambda^{^{\prime
}}}(z)=e_{\lambda+\lambda^{^{\prime }}}(z), \hspace*{10pt} \forall \
\lambda, \lambda^{^{\prime }}\in \Lambda . \vspace*{8pt}
\end{equation}
It is known that $c_{1}(L_{0}) = [\omega]$, $\omega = \delta \ dx_{1} \wedge
dx_{2}$.

This description helps to give an explicit basis of global sections. More
precisely, write $\pi : V \rightarrow M= V / \Lambda$ for the projection.
There is a trivialization $\phi : \pi^{*}L_{0} \rightarrow V \times \mathbb{C%
}$ of $\pi^{*}L_{0}$ such that for any global holomorphic section $\tilde{%
\theta}$ of $L_{0} \rightarrow M$, the function $\theta := (\phi^{-1})^{*}
(\pi^{*} \tilde{\theta})$ is a quasi-periodic entire function on $V$
satisfying 
\begin{equation}  \label{eqn3}
\begin{cases}
\theta (z+ \lambda_{1})= \theta(z) \\ 
\theta (z+ \lambda_{2})= e^{- 2 \pi i z - \pi i \tau}\ \theta(z), \hspace*{%
40pt} \forall z \in V.%
\end{cases}
\hspace*{50pt}
\end{equation}

\noindent By the same token, a Hermitian metric $h_{L_{0}}(z)>0$ on $L_{0}$
where 
\begin{equation*}
||\tilde{\theta}(\pi (z))||_{h_{L_{0}}}^{2}:=h_{L_{0}}(z)\,|\theta (z)|^{2},
\end{equation*}%
is also characterized by the quasi-periodic property:\newline
\begin{equation}
\begin{cases}
h_{L_{0}}(z+\lambda _{1})=h_{L_{0}}(z) \\ 
h_{L_{0}}(z+\lambda _{2})=e^{-4\pi z_{2}-2\pi \tau _{2}}\,h_{L_{0}}(z),%
\hspace*{15pt}\forall z\in V.%
\end{cases}%
\hspace*{55pt}\vspace*{12pt}  \label{eqn4}
\end{equation}

\begin{lemma}
\label{Lemma 1} For the holomorphic line bundle $L_{0} \rightarrow M$, one
can use the quasi-periodic entire functions on $V$ \newline
\begin{equation}
\theta_{m}(z) :=\underset{k \in \mathbb{Z}}{\Sigma}\, e^{\pi i k^{2}
\tau}e^{2 \pi i \tau \frac{m}{\delta} k} e^{2 \pi i \frac{(k \delta + m)}{%
\delta} z }, \ \mbox{\small \  $m=0, 1,..., \delta -1,$}
\end{equation}
as a basis of global holomorphic sections of $L_{0}$, and \newline
\begin{equation}  \label{hLo}
h_{L_{0}}(z) := e^{\frac{- 2 \pi}{\tau_{2}}(z_{2})^{2}} \hspace*{160pt}
\end{equation}
as a metric on $L_{0}$.
\end{lemma}

\begin{proof}
For the special case $\delta=1$, $m=0$
\begin{equation}
\theta_{0}(z) =\underset{k \in \mathbb{Z}}{\Sigma}\, e^{\pi i k^{2} \tau}  e^{2 \pi i k z } = \vartheta (z, \tau) \hspace*{20pt} \forall z \in V
\end{equation}
is the Riemann theta function. For general $\delta \in \mathbb{N}$, $m=0,...,\delta-1$,
\begin{equation}
\theta_{m}(z) = e^{\tiny 2 \pi i \frac{m}{\delta} z}\  {\vartheta ( z+ \frac{m}{\delta}\tau,\, \tau)} \hspace*{30pt} \forall z \in V
\end{equation}
is a translate of $\vartheta(z, \tau)$ multiplied by the exponential factor $e^{\tiny 2 \pi i \frac{m}{\delta} z}$. The lemma follows easily from $(\ref{eqn3})$, $(\ref{eqn4})$ and the quasi-periodic property of the Riemann theta function.
\end{proof}

For any $\mu \,e_1\in V$, $\mu=\mu_{1}+ i \mu_{2}$, we have a map\newline
\begin{equation*}
\mathcal{T}_{\mu} : M \mapsto M \vspace*{8pt}
\end{equation*}
defined by the translation by $[\mu]\in M$. Let $L_{\mu} := \mathcal{T}%
^{*}_{\mu}L_{0} \rightarrow M $. Then $L_{\mu}$ can be given by multipliers 
\newline
\begin{equation}  \label{eqnL}
\begin{cases}
e_{\lambda_{1}}(z) \equiv 1 \\ 
e_{\lambda_{2}}(z) = e^{- 2 \pi i z - 2 \pi i \mu- \pi i \tau}.%
\end{cases}
\hspace*{155pt} \vspace*{8pt}
\end{equation}

\noindent In the same vein as before, any global holomorphic sections $%
\tilde{\theta}$ of $L_{\mu} \rightarrow M$ can be described via
quasi-periodic entire functions $\theta$ on $V$ satisfying \newline
\begin{equation}  \label{theta section}
\begin{cases}
\theta (z+ \lambda_{1})= \theta(z) \\ 
\theta (z+ \lambda_{2})= e^{- 2 \pi i ( z + \mu) - \pi i \tau}\ \theta(z), 
\hspace*{20pt} \forall z \in V,%
\end{cases}
\hspace*{60pt} \vspace*{8pt}
\end{equation}
and the metric $h_{L_{\mu}}(z)$ on $L_{\mu} \rightarrow M$: 
\begin{equation}  \label{hLmu}
\begin{cases}
h_{L_{\mu}} (z+ \lambda_{1})= h_{L_{\mu}}(z) \\ 
h_{L_{\mu}} (z+ \lambda_{2})= e^{-4 \pi ( z_{2} + \mu_{2} ) - 2 \pi
\tau_{2}}\ h_{L_{}\mu} (z),\hspace*{10pt} \forall z \in V.%
\end{cases}
\hspace*{30pt} \vspace*{8pt}
\end{equation}

It is well known that all the holomorphic line bundles on $M$ having the
same first Chern class as $L_{0}$ can be represented as a translate of $%
L_{0} $. As a consequence, by Lemma $\ref{Lemma 1}$, $(\ref{theta section})$
and $(\ref{hLmu})$, one has: \newline

\begin{lemma}
\label{Lemma 2} Fix a $\mu\in V$. For the holomorphic line bundle $L_{\mu}
\rightarrow M$ as defined above, one can use the quasi-periodic entire
functions on $V:$ 
\begin{equation}
\theta_{m}(z, \mu)= \theta_{m}(z+\mu ) = \underset{k \in \mathbb{Z}}{\Sigma}%
\, e^{\pi i k^{2} \tau}e^{2 \pi i \tau \frac{m}{\delta} k} e^{2 \pi i \frac{%
(k \delta + m)}{\delta} (z+\mu) }, \ 
\mbox{\small \  $m=0, 1,..., \delta
-1$,}
\end{equation}
as a basis of global holomorphic sections of $L_{\mu}$, and 
\begin{equation}  \label{lemhL}
h_{L_{\mu}}(z) = h_{L}(z+ \mu) = e^{\frac{- 2 \pi}{\tau_{2}}(z_{2} +
\mu_{2})^{2}} \hspace*{120pt}
\end{equation}
as a metric on $L_{\mu}$.
\end{lemma}


\section{The dual torus $\protect\widehat{M} = \mathrm{Pic}^{0} (M)$ of $M$}

\label{section 2}

The notational convention here follows that of \cite[p. 307-317]{GH} unless
specified otherwise. We have a natural identification for the set $\mathrm{%
Pic}^{0} (M)$ : 
\begin{equation}  \label{Pic0}
\mathrm{Pic}^{0}(M) \cong \frac{H^{1}(M, \mathcal{O})}{H^{1}(M, \mathbb{Z})}
\cong \frac{H_{\overline{\partial}}^{0,1}(M)}{H^{1}(M, \mathbb{Z})}
\end{equation}
via the long exact cohomology sequence associated with the exponential sheaf
sequence for the first isomorphism, and the Dolbeault isomorphism for the
second, where the map $H^{1}(M, \mathbb{Z}) \rightarrow H_{\overline{\partial%
}}^{0,1}(M)$ is given by 
\begin{equation*}
\omega \mapsto \omega^{0,1}. \vspace*{8pt}
\end{equation*}

\noindent The image of $H^{1}(M, \mathbb{Z})$ in ${\overline{V}}^{\, *}= H_{%
\overline{\partial}}^{0,1}(M)$ is the lattice ${\overline{\Lambda}}^{\, *}= 
\mathbb{Z}\, \{\, dx_{1}^{*},\, dx_{2}^{*}\, \}$ which consists exactly of 
\textit{conjugate linear functionals} on $V$ whose real part is
half-integral on $\Lambda \subseteq V$. See below. $\mathrm{Pic}^{0}(M) = {\overline{V%
}}^{\, *} /\, {\overline{\Lambda}}^{\, *}$ is often called the dual torus of 
$M$, and denoted as $\widehat{M}$. 

To be precise, we write the \textit{conjugate linear part} of $dx_{1}$%
, $dx_{2}$ as \newline
\begin{equation}
\begin{cases}
{dx_{1}}^{*} = \overline{\Pi}_{11}\, d \overline{z} = \frac{1}{2 \delta} (1
- i \frac{\tau_{1}}{ \tau_{2}})\, d \overline{z} \\ 
{dx_{2}}^{*} = \overline{\Pi}_{21}\, d \overline{z} = \frac{i}{2 \tau_{2}}\,
d \overline{z}%
\end{cases}
\hspace*{120pt} \vspace*{8pt}
\end{equation}

\noindent from (cf. (\ref{2.0}))\newline
\begin{equation}
\begin{cases}
dx_{1}=\Pi _{11}\,dz+\overline{\Pi }_{11}\,d\overline{z}=\frac{1}{2\delta }%
(1+i\frac{\tau _{1}}{\tau _{2}})\,dz+\frac{1}{2\delta }(1-i\frac{\tau _{1}}{%
\tau _{2}})\,d\overline{z} \\ 
dx_{2}=\Pi _{21}\,dz+\overline{\Pi }_{21}\,\overline{dz}=\frac{-i}{2\tau _{2}%
}\,dz+\frac{i}{2\tau _{2}}\,d\overline{z}.%
\end{cases}%
\vspace*{8pt}  \label{x-zcoord}
\end{equation}%

\noindent Re-ordering $\{{dx_{1}}^{\ast },{dx_{2}}^{\ast }\}$
we set 
\begin{equation*}
{dy_{1}}^{\ast }=-{dx_{2}}^{\ast },\hspace*{16pt}{dy_{2}}^{\ast }={dx_{1}}%
^{\ast }=\frac{\tau }{\delta }\,{dy_{1}}^{\ast }.
\end{equation*}

\noindent Setting ${e_{1}}^{\ast }:={dy_{1}}^{\ast }$, we have the lattice 
\begin{equation*}
{\overline{\Lambda }}^{\,\ast }=\mathbb{Z}\,\{{dy_{1}}^{\ast },{dy_{2}}%
^{\ast }\}=\mathbb{Z}\,\{\,{e_{1}}^{\ast },\frac{\tau }{\delta }\,{e_{1}}%
^{\ast }\}.
\end{equation*}%

One has the map $\varphi_{L_{0}} : M \rightarrow \mathrm{Pic}^{0}(M)$%
\ defined, via the translation ${\mathcal{T}}_{\mu} : M \rightarrow M$ with 
$[\mu ] \in M$, by
 
\begin{equation*}
\varphi_{L_{0}}([\mu]) = {{\mathcal{T}}_{\mu}}^{*}L_{0} \otimes {L_{0}}^{*}, 
\hspace*{30pt} \forall \mu \in V, \vspace*{8pt}
\end{equation*}

\noindent and the natural lifting map\, $\widetilde{\varphi_{L_{0}}} : V \rightarrow {%
\overline{V}}^{\, *}$ of $\varphi_{L_{0}}$. In general, $\varphi_{L_{0}}$ is not an isomorphism
unless $\delta =1$. 

The following property is well-known:

\begin{property}
\label{prop1}$\widetilde{\varphi_{L_{0}}} : V\rightarrow {%
\overline{V}}^{\, *}$ is a complex linear transformation such that 
\begin{equation}
\widetilde{\varphi_{L_{0}}}(e_{1} ) = e_{1}^{*}  \label{eqn 17}
\end{equation}
\end{property}

\begin{proof}
Let us go back to the map
\begin{equation}
H_{\overline{\partial}}^{0,1}(M) \overset{\delta}{\longrightarrow} H^{1}(M, \mathcal{O}) \overset{p}{\longrightarrow} \frac{H^{1}(M, \mathcal{O})}{H^{1}(M, \mathbb{Z})}  \cong  {\rm Pic}^{0}(M)
\end{equation}
where $\delta$ is the Dolbeault isomorphism and $p$ is the projection. For any $\alpha =\sigma\, d \overline{z} \in H_{\overline{\partial}}^{0,1}(M)$,\ $p \circ \delta$ sends $\alpha$ to the line bundle given by the multipliers
\begin{equation} \label{eqn19}
\begin{cases}  e_{\lambda_{1}}(z) = e^{2 \pi i \sigma \delta} \\ e_{\lambda_{2}}(z) = e^{2 \pi i \sigma \overline{\tau}}. \end{cases} \hspace*{170pt}
\end{equation}

\noindent  Note that this choice of line bundles is dual to
the one given in \cite[p. 315-316]{GH}.  

Multiplying the trivializations by the function $f(z)=e^{\,- 2 \pi i \sigma z}$ yields the normalized multipliers
\begin{equation} \label{eqn20}
\begin{cases}  e_{\lambda_{1}}(z) \equiv 1 \\ e_{\lambda_{2}}(z) = e^{4 \pi \sigma \tau_{2}}. \end{cases} \hspace*{170pt}
\end{equation}

\noindent  On the other hand, the multipliers of ${{\mathcal{T}}_{\mu}}^{*}L_{0} \otimes {L_{0}}^{*}$
are, via $(\ref{eqnL0})$ and $(\ref{eqnL})$,
\begin{equation}  \label{eqn21}
\begin{cases}  e_{\lambda_{1}}(z) \equiv 1 \\ e_{\lambda_{2}}(z) = e^{- 2 \pi i \mu}. \end{cases} \hspace*{170pt}
\end{equation}

\noindent Plugging $\mu =1$ ($\mu e_{1}= e_{1}$) into $(\ref{eqn21})$ and setting
$\alpha =e_{1}^{*}= \frac{-i}{2 \tau_{2}}\, d\overline{z}=\sigma\, d\overline{z}$ in $(\ref{eqn20})$, one obtains

\begin{equation*}
(\ref{eqn20})=(\ref{eqn21}),
\end{equation*}

\noindent hence $(\ref{eqn 17})$.  We omit the proof that  $\widetilde{\varphi_{L_{0}}}$ is complex linear.
\end{proof}

We should also recall the \textbf{Poincar\'e line bundle}.
Let $\hat{\mu}=\hat{\mu}_{1}+ i \hat{\mu}_{2}$ be the complex coordinate on $%
{\overline{V}}^{\, *}$ (and on $\widehat{M}$) such that $d \hat{\mu}$ is
dual to $e_{1}^{*}$. As previously, an element $\hat{\mu}\, e_{1}^{*}\in {%
\overline{V}}^{\, *}$ is interchangeably written as $\hat{\mu} \in {%
\overline{V}}^{\, *}$. We denote the line bundle corresponding to $[\hat{\mu}%
] \in \widehat{M} = \mathrm{Pic}^{0}(M)$ in $(\ref{Pic0})$ by $\mathrm{P}_{[\hat{\mu}]}$ or $%
\mathrm{P}_{\hat{\mu}}$ if there is no danger of confusion. By Property $\ref{prop1}$
above, we can also write 
\begin{equation}  \label{Pmu}
\mathrm{P}_{\hat{\mu}}=\mathrm{P}_{\varphi_{L_{0}}([\mu])} \cong {{\mathcal{T}}_{\mu}}%
^{*}L_{0} \otimes {L_{0}}^{*}, \hspace*{20pt} \forall\, \mu \in V, \vspace*{%
8pt}
\end{equation}
where $\mu$ and $\hat\mu$ are related by $\widetilde{\varphi_{L_{0}}}(\mu\,
e_1)=\hat\mu\, e_1^*$. The following lemma is standard. 

\begin{lemma}
\label{Poincare1} There is a unique holomorphic line bundle $%
\mathrm{P} \rightarrow M \times \widehat{M}$ called the Poincar\'e line bundle
satisfying : \vspace*{4pt} \newline
$(1)\ \mathrm{P}_{\mid_{M \times \{ \hat{\mu} \} } } \cong \mathrm{P}_{\hat{\mu} }$.\vspace*{%
-4pt} \newline
$(2)\ \mathrm{P}_{\mid_{\{0\} \times \widehat{M}}}$ is a holomorphically trivial line
bundle. 
\end{lemma}

\section{A holomorphic line bundle $\protect\widetilde{K}
\rightarrow M_{1} \times M_{2} \protect\cong M \times M$}

\label{section 3} 

As explained in the second half of Introduction, we would
like to "accomodate" the $\mu$-dependent theta functions $\theta_{m}(z,\mu )$ of previous sections.
For this need, we introduce an intermediate line bundle $\widetilde{K}$ in
this section. Let $M_{1} \cong M_{2} \cong M$ and $\pi : V \times V
\rightarrow M_{1} \times M_{2} = V / \Lambda \times V / \Lambda $ with
projections $\pi_{i}: M_{1} \times M_{2} \rightarrow M_{i},\ i=1,2.$ We
denote by $\lambda_{10}=\lambda_1$ and $\lambda_{20}=\lambda_2$ for the
first lattice and $\lambda_{01}=\lambda_1$ and $\lambda_{01}=\lambda_2$ for
the second one. Recall the map $\varphi_{L_{0}}:M\to \hat M$ in the
preceding section, and form $Id \times \varphi_{L_{0}}:M\times M\to
M\times\hat M$. 

\begin{definition}
\label{def 1} We define the holomorphic line bundle $%
\widetilde{K} : = \pi^{*}_{1} L_{0} \otimes (Id \times \varphi_{L_{0}})^{*}\mathrm{P}
\otimes \pi^{*}_{2} L_{0} \rightarrow M_{1} \times M_{2}$ where $\mathrm{P} \rightarrow
M \times \hat{M}$ is the Poincar\'e line bundle. 
\end{definition}

\begin{proposition}
\label{proposition 1} In notations as above, a system of
multipliers of $\widetilde{K}$ can be 
\begin{equation}  \label{multiK}
\begin{cases}
e_{\lambda_{10}}(z, \mu)\equiv 1 , & e_{\lambda_{01}}(z, \mu)\equiv 1 \\ 
e_{\lambda_{20}}(z, \mu) = e^{- 2 \pi i z - 2 \pi i \mu- \pi i \tau}, & 
e_{\lambda_{02}}(z, \mu) = e^{- 2 \pi i z - 2 \pi i \mu- \pi i \tau}.%
\end{cases}%
\end{equation}
\end{proposition}

\begin{proof}
Recall that a holomorphic line bundle on $M \times M= V / \Lambda \times V / \Lambda $
is essentially described by a set of data: a system of multipliers $\{ e_{\lambda_{10}}$, $e_{\lambda_{20}}$, $e_{\lambda_{01}}$, $e_{\lambda_{02}} \in \mathcal{O}^{*}(V \times V) \}$ satisfying the \textbf{compatibility relations} (cf.(\ref{compatibility})) : for $\{i,j\}=\{1,2\}$
\begin{equation}  \label{compatibility2}
\begin{cases}  e_{\lambda_{i0}}(z+\lambda_{j}, \mu)\, e_{\lambda_{j0}}(z, \mu)  = e_{\lambda_{j0}}(z+\lambda_{i}, \mu) \,  e_{\lambda_{i0}}(z, \mu) \\
e_{\lambda_{i0}}(z, \mu+\lambda_{j})\, e_{\lambda_{0j}}(z, \mu)= e_{\lambda_{0j}}(z+\lambda_{i}, \mu)\, e_{\lambda_{i0}}(z, \mu) \\
e_{\lambda_{0i}}(z, \mu +\lambda_{j})\, e_{\lambda_{0j}}(z, \mu)= e_{\lambda_{0j}}(z, \mu +\lambda_{i})\, e_{\lambda_{0i}}(z, \mu).\\
\end{cases}
\end{equation}

  To break things down, the multipliers of $\pi_{1}^{*}L_{0}$ can be
 \begin{equation} \label{multipi1L0}
\begin{cases}  e_{\lambda_{10}}(z, \mu) \equiv 1,
&e_{\lambda_{20}}(z, \mu) = e^{- 2 \pi i z - \pi i \tau} \\
e_{\lambda_{01}}(z, \mu) \equiv 1,
&e_{\lambda_{02}}(z, \mu) \equiv 1  \end{cases}  \hspace*{80pt}
\end{equation}

\noindent and the multipliers of $\pi_{2}^{*}L_{0}$ can be similarly expressed. As we will see soon, a system of
multipliers of $(Id \times \varphi_{L_{0}})^{*}{\rm P}$ can be chosen to be
\begin{equation}  \label{idphiP}
\begin{cases}  e_{\lambda_{10}}(z, \mu) \equiv 1,
&e_{\lambda_{20}}(z, \mu) = e^{-2 \pi i \mu} \\
e_{\lambda_{01}}(z, \mu) \equiv 1,
&e_{\lambda_{02}}(z, \mu) = e^{-2 \pi i z}. \end{cases}  \hspace*{90pt}
\end{equation}
Obviously all these multipliers satisfy $(\ref{compatibility2})$.  So
$(\ref{idphiP})$ does define a holomorphic line bundle, tentatively denoted by $J$, on $M_1\times M_2$.

  To see the above claim $(\ref{idphiP})$, note first that a system of multipliers of
  
$$(Id \times \varphi_{L_{0}})^{*}{\rm P}_{\mid_{M \times \{\mu\}}} \cong {{\mathcal{T}}_{\mu}}^{*}L_{0}
 \otimes {L_{0}}^{*} \cong {\rm P}_{\varphi_{L_{0}}(\mu)} \rightarrow M_{1} \times M_{2}$$ 
 
\noindent can be

\begin{equation} \label{multiMmu}
\begin{cases}  e_{\lambda_{1}}(z) \equiv 1 \\
e_{\lambda_{2}}(z) = e^{-2 \pi i \mu}  \end{cases}  \hspace*{185pt}
\end{equation}

\noindent and that of the trivial line bundle $(Id \times \varphi_{L_{0}})^{*}{\rm P}_{\mid_{\{0 \} \times M}}$

\begin{equation} \label{multi0M}
 e_{\lambda_{1}}(\mu) =
e_{\lambda_{2}}(\mu) \equiv 1.   \hspace*{160pt} \vspace*{4pt}
\end{equation}

\noindent One observes that $(\ref{idphiP})$ or $J$ satisfies $(\ref{multiMmu})$ and $(\ref{multi0M})$. The claim that

$$J\cong (Id \times \varphi_{L_{0}})^{*}{\rm P}$$ 

\noindent follows from the same type of arguments of \cite[p. 329]{GH}
for the proof of the uniqueness of Poincar\'e line bundle.  Our claim $(\ref{idphiP})$ is proved.

Finally, $(\ref{multiK})$ follows from $(\ref{multipi1L0})$ (for $\pi_1^*L_0$ and similarly for $\pi_2^*L_0$)
and $(\ref{idphiP})$.
\end{proof}

By Proposition $\ref{proposition 1}$, any global
holomorphic sections $\tilde{\theta}$ of $\widetilde{K} \rightarrow M_{1}
\times M_{2}$ can be represented by quasi-periodic holomorpic functions on $%
V \times V$ satisfying, for all $z, \mu \in V$,
 
\begin{align}  \label{K}
\begin{cases}
\theta (z+ \lambda_{1}, \mu)= \theta(z, \mu) = \theta (z, \mu + \lambda_{1})
\\ 
\theta (z+ \lambda_{2}, \mu)= e^{- 2 \pi i z -2 \pi i \mu - \pi i \tau}\
\theta(z, \mu)= \theta (z, \mu + \lambda_{2})%
\end{cases}
\end{align}

\noindent and any Hermitian metric $h(z, \mu)$ on $\widetilde{K} \rightarrow
M_{1} \times M_{2}$: 
\begin{align}  \label{K metric}
\begin{cases}
h(z+ \lambda_{1}, \mu)= h(z, \mu)=h(z, \mu + \lambda_{1}) \\ 
h (z+ \lambda_{2}, \mu)= e^{-4 \pi z_{2} - 4 \pi \mu_{2} - 2 \pi \tau_{2}}\,
h(z, \mu) = h (z, \mu + \lambda_{2} ).%
\end{cases}
\\
\notag
\end{align}

An application of Proposition \ref{proposition 1} is to
exploit those $\mu$-dependent theta functions $\theta_{m}(z, \mu)$. Recall
that in Lemma \ref{Lemma 2}, $\{ \theta_{m}(z, \mu) \}_{m}$ represents a
basis of the global holomorphic sections of $L_{\mu}$ for each individual $%
\mu \in V$. As $\mu$ varies, it seems tempting to think that $\{
\theta_{m}(z, \mu) \}_{m}$ naturally extends the sections $\{ \theta_{m}(z,
0) \}_{m}$ of $L_0$ via the Poincar\'e line bundle along the $\mu$%
-direction. This is not quite the case, however. 

Indeed, a global property that this family of functions $\{
\theta_{m}(z, \mu) \}_{m}$ possess is the following. \vspace*{8pt} 

\begin{theorem}
\label{theorem 1} For the holomorphic line bundle $%
\widetilde{K} = \pi^{*}_{1} L_{0} \otimes (Id \times \varphi_{L})^{*}\mathrm{P}
\otimes \pi^{*}_{2} L_{0} \rightarrow M_{1} \times M_{2}$, one has the
quasi-periodic holomorphic functions on $V \times V$ 
\begin{equation}  \label{K formula}
\theta_{m}(z, \mu) = \underset{k \in \mathbb{Z}}{\Sigma}\, e^{\pi i k^{2}
\tau}e^{2 \pi i \tau \frac{m}{\delta} k} e^{2 \pi i \frac{(k \delta + m)}{%
\delta} (z+\mu) }, \ \mbox{\small \  $m=0, 1,..., \delta -1$,}
\end{equation}
as a basis of global holomorphic sections of $\widetilde{K}$, and 
\begin{equation}  \label{K metric formula}
h(z, \mu) = e^{\frac{- 2 \pi}{\tau_{2}}(z_{2} + \mu_{2})^{2}} \hspace*{120pt}
\end{equation}
as a metric on $\widetilde{K}$, which on the restriction $\widetilde{K}%
_{\mid_{M \times \{ \mu \}}}$ induces the metric $h_{L_{\mu}}$ in $(\ref%
{lemhL})$. \vspace*{8pt} 
\end{theorem}

\begin{proof}
  Let $\omega := z + \mu$. By using the quasi-periodic property  of $\theta_{m}(\omega)$ and $h_{L_{0}}(\omega)$ in  $(\ref{theta section})$ and $(\ref{hLmu})$, we see that the functions $(\ref{K formula})$ and $(\ref{K metric formula})$ satisfy $(\ref{K})$ and $(\ref{K metric})$. The theorem follows. 
\end{proof}

We can now equip the line bundle 
\begin{equation*}
(Id \times \varphi_{L_{0}})^{*}\mathrm{P} \rightarrow M_{1} \times M_{2}
\end{equation*}
with a metric. Since $\widetilde{K} = \pi^{*}_{1} L_{0} \otimes (Id \times
\varphi_{L})^{*}\mathrm{P} \otimes \pi^{*}_{2} L_{0} \rightarrow M_{1} \times M_{2}$,
by the metric $h(z, \mu)$ on $\widetilde{K}$ (cf. $(\ref{K metric formula})$%
) and the metric $h_{L_{0}}(z)$ (cf. $(\ref{hLo})$), one finds the induced
metric 
\begin{equation}  \label{metric IdPhiP}
h_{(Id \times \varphi_{L_{0}})^{*}\mathrm{P}}(z, \mu)= e^{-\frac{4 \pi }{\tau_{2}}%
z_{2} \mu_{2}}
\end{equation}
on $(Id \times \varphi_{L_{0}})^{*}\mathrm{P}$. Let's now calculate the curvature of
this metric. 

\begin{theorem}
\label{theorem 2} The curvature of the metric in $(\ref%
{metric IdPhiP})$ is 
\begin{equation}  \label{curvP}
\Theta_{(Id \times \varphi_{L_{0}})^{*}\mathrm{P}} (z, \mu) = \frac{\pi}{\tau_{2}}\, %
\big( dz \wedge d \overline{\mu} + d \mu \wedge d \overline{z} \big).
\end{equation}
\end{theorem}

\begin{proof}
  The curvature $\Theta_{\widetilde{K}}$ of $(\widetilde{K}, h(z, \mu))\ $  is
\begin{align} \label{curv1}
\Theta_{\widetilde{K}}(z, \mu) &= -\, \partial \, \overline{\partial}\, log (h(z, \mu)) \nonumber \\
& = \frac{\pi}{\tau_{2}}\, \big(\,  dz \wedge d \overline{z} + d \mu \wedge d \overline{ \mu } + dz \wedge d \overline{\mu} + d \mu \wedge d \overline{z}\, \big)
\end{align}
\noindent and the curvature $\Theta_{L_{0}}(z)$ of $\big( L_{0}, h_{L_{0}}(z) \big)$ is
\begin{equation} \label{curv2}
\Theta_{L_{0}}(z) = -\, \partial \, \overline{\partial}\, log (h_{L_{0}}(z)) = \frac{\pi}{\tau_{2}}\, dz \wedge d \overline{z}. \end{equation}
Now $(\ref{curvP})$ follows from $(\ref{curv1})$ and $(\ref{curv2})$.
\end{proof}

\section{The holomorphic vector bundle ${K} \rightarrow
M_{2} \, \protect\cong M $}

\label{section 4} 

To facilitate the curvature computation later on, we shall
now discuss the direct image bundle $K$ of $\widetilde{K}$ in the preceding
section. Recalling the line bundle $\widetilde{K} \rightarrow M_{1} \times
M_{2}$ (cf. Definition \ref{def 1}), we form the push-forward ${K} := {%
\pi_{2}}_{*}\widetilde{K}$ which is a holomorphic vector bundle on $M_{2}$.
One sees that ${K}= {\pi_{2}}_{*} \big( \pi^{*}_{1} L_{0} \otimes (Id \times
\varphi_{L})^{*}\mathrm{P} \big) \otimes L_{0} $ on $M_2$ by the standard projection
formula.\newline

\begin{definition}
\label{def 2} Define a metric $\big(\ \ , \ \ \big)_{h}$ on $%
{K}$ by the $L^{2}$ inner product using $\big(\ \ , \ \ \big)_{h_{L_{\mu}}}$
on ${K}_{\mid{\mu}} = H^{0} (M, \widetilde{K}_{\mid_{M \times \{ \mu \}}})$
(cf. the last statement in Theorem \ref{theorem 1}): 
\begin{equation}  \label{metric h}
\big( \theta (z) , \theta^{\prime} (z) \big)_{h_{L_{\mu}}} := \int_{M}
h_{L_{\mu}}(z)\ \theta (z)\ \overline{ \theta^{\prime} (z)}\ (\, \frac{i}{2}
dz \wedge d \overline{z} \,)
\end{equation}
where $\theta$, $\theta^{\prime}$ are global holomorphic sections of $%
L_{\mu} $. \vspace*{8pt} 
\end{definition}

The main lemma for our computations is as follows. 

\begin{lemma}
\label{Lemma 4} With the inner product $( \ \ ,\ \ )_{h_{{%
L_\mu}}}$, the holomorphic sections 
\begin{equation}
\theta_{m}(z, \mu) = \underset{k \in \mathbb{Z}}{\Sigma}\, e^{\pi i k^{2}
\tau}e^{2 \pi i \tau \frac{m}{\delta} k} e^{2 \pi i \frac{(k \delta + m)}{%
\delta} (z+\mu) }, \ \mbox{\small \  $m=0, 1,..., \delta -1$},
\end{equation}
constitute an \textbf{orthogonal} basis of $H^{0}(M, L_{\mu})$, where 
\begin{equation}
\big(\theta_{m}(z, \mu), \theta_{m}(z, \mu)\big)_{h_{L_{\mu}}} = \sqrt{ 
\frac{\tau_{2}}{2} }\ \delta \ e^{ \frac{2 \pi m^{2}}{\delta^{2}} \tau_{2}
}, \ \mbox{\small \  $m=0, 1,..., \delta -1$.}
\end{equation}
\end{lemma}

\begin{proof}
  By $(\ref{metric h})$, we have 
  
\begin{multline} \label{integral1}
\big( \theta_{m} (z, \mu) , \theta_{m'} (z, \mu) \big)_{h_{L_{\mu}}} = \int_{M} h_{L_{\mu}}(z)\ \theta_{m} (z, \mu)\ \overline{ \theta_{m'} (z, \mu)}\ (\, \frac{i}{2} dz \wedge d \overline{z} \,) \\
= \int_{0}^{\tau_{2}}\int_{0}^{\delta} \underset{k,j \in \mathbb{Z}}{\Sigma}\, e^{\frac{- 2 \pi}{\tau_{2}}(z_{2} + \mu_{2})^{2}} \Big(  e^{\pi i k^{2} \tau}e^{2 \pi i \tau \frac{m}{\delta} k} e^{2 \pi i (k+\frac{m}{\delta}) (z+\mu) } \Big)\\
\Big(  e^{-\pi i j^{2} \overline{\tau}}e^{-2 \pi i \overline{\tau} \frac{m'}{\delta} j} e^{-2 \pi i (j+\frac{m'}{\delta}) (\overline{z}+\overline{\mu}) } \Big)\, dz_{1}\, dz_{2} \hspace*{20pt}
\end{multline}
\noindent where $z=z_{1}+ i z_{2}$, $z_{1}, z_{2} \in \mathbb{R}$.  The terms in $(\ref{integral1})$ related to $z_{1}$ are
\begin{equation}
 \int_{0}^{\delta} e^{2 \pi i z_{1}(k-j + \frac{m-m'}{\delta})}\, dz_{1}
\end{equation}
which survive only when $k=j$ and $m=m'$. The lemma follows by straightforward calculations
in the following aspects:
\begin{itemize}
\item[i)] change of variable $t:= \frac{1}{\tau_{2}}\, (z_{2}+ \mu_{2})$,
\item[ii)] the union of the domains of definite integrals
\[
\underset{k \in \mathbb{Z}}{\Sigma} \int_{\frac{\mu_{2}}{\tau_{2}}}^{1+\frac{\mu_{2}}{\tau_{2}}}\, e^{-2 \pi \tau_{2}(t+k+\frac{m}{\delta})^{2}}\, dt = \int_{-\infty}^{\infty}\, e^{- 2 \pi \tau_{2}\, t^{2}}\, dt,
\]
\item[iii)] the Gaussian integral (where we use $A= 2 \pi \tau_{2}$)
\[
\int_{\infty}^{\infty} e^{-A t^{2}}\, dt = \frac{\sqrt{\pi}}{\sqrt{A}},  \hspace*{20pt} A > 0.
\]
\end{itemize}
\end{proof}

By this lemma, the value of $\big(\theta_{m}(z, \mu),
\theta_{m}(z, \mu)\big)_{h_{L_{\mu}}}$ in Definition $\ref{def 2}$ is
independent of $\mu$. We obtain the first statement of the following
theorem. \newline

\begin{theorem}
\label{KK} $(1)$\ On ${K}$, the curvature tensor of the
metric $(\ \ , \ \ )_{h}$ defined in Definition \ref{def 2}, is identically zero. 
\newline
$(2)$ \ ${K}$ splits holomorphically into a direct sum of holomorphically
trivial line bundles ${K} = \overset{\delta -1}{\underset{m=0}{\bigoplus}} {K%
}_{m} $where each ${K}_{m}$ has the canonical section identified as $%
\theta_m $ of Lemma \ref{Lemma 4}. 
\end{theorem}

\begin{proof}
  The first statement is observed precedingly; the second statement follows from
Theorem \ref{theorem 1}, Lemma \ref{Lemma 4} and the first statement.
\end{proof}

\begin{remark}
\label{rem3r} For the above second statement, there is an
argument without using metric. Since $M$ is of dimension one, each $\theta_m$
of Lemma \ref{Lemma 4} generates a holomorphic line subbundle of $K\to
M_{2}\cong M$, still denoted by $K_m \to M_{2}$. It is not difficult to see
that $\theta_m$ is actually nowhere vanishing on $M_{2}$ by using the fact
that by construction, it arises from translates of the ordinary theta
functions. Hence $K_m$ is holomorphically trivial. By similar arguments, $%
\{\theta_m\}_m$ is also independent everywhere on $M$ and hence a global
basis for $K\to M$.
\end{remark}

\section{Connection on the line bundle $\mathcal{P}
\rightarrow M \times M^{*}$}

\label{section 5} 

The vector bundle to be computed is going to live on $\hat M$%
. For this reason and others as explained earlier in Introduction, we are
led to differential geometric aspects of the Poincar\'e line bundle in this
section and the next one. Here, we view the Riemann surface $M$ as a real
2-dimensional smooth manifold and introduce a differential geometric
description of the Poincar\'e line bundle with a connection on it. We follow
closely the treatment in \cite[Subsections 3.2.1 and 3.2.2]{DK}, but use a
suitable sign convention more adapted to our purpose. 

To begin with, we write $V \cong \mathbb{R}^{2}$, and $M= V
/ \Lambda$\, where $\Lambda = \{ \lambda_{1}, \lambda_{2} \}= \{ (\delta,
0),\ (\tau_{1}, \tau_{2}) \}$, $\delta \in \mathbb{N}$, $\tau_{2} > 0$. Let $%
\Lambda^{*}=\{ dx_{1}, dx_{2} \}$ be the dual basis of $\Lambda$; that is, $%
\int_{\lambda_{i}}\, dx_{j}= \delta_{ij}$. 
Let $V^{*}:=Hom(V, \mathbb{R})$ be the dual space
of $V$. Any $\xi \in V^{*}$ is a 1-form with constant real coefficients.
That is, $\xi = \xi_{1}\, dx_{1} + \xi_{2}\, dx_{2}$ with $\xi_{1}$, $%
\xi_{2} \in \mathbb{R}$. We define 
\begin{equation}  \label{M*}
M^{*} := V^{*} / \, 2 \pi \Lambda^{*},
\end{equation}
\noindent and write $[\xi]$ as the equivalent class of $\xi$ in $M^{*}$. 

Let $\underline{\mathbb{C}\, }_{\mid_{V}} : V \times \mathbb{%
C} \rightarrow V$ be the trivial complex line bundle over $V$. An element $%
\xi \in V^{*}$ gives rise to a \textbf{character} $\chi_{\xi} : \Lambda
\rightarrow U(1)$ by 
\begin{equation}
\chi_{\xi}(\lambda) := e^{-i\, <\xi, \lambda>}
\end{equation}
where $<\xi, \lambda> = \xi(\lambda) \in \mathbb{R}$. The set $\Lambda$ acts
on $\underline{\mathbb{C}\, }_{\mid_{V}}$ by 
\begin{equation}  \label{actionlambda}
\lambda \circ (x, \sigma) := (x+\lambda,\, \chi_{\xi}(\lambda)\, \sigma).
\end{equation}
This action preserves the horizontal foliation in $\underline{\mathbb{C}\, }%
_{\mid_{V}}$ which thus descends to a flat connection, denoted by $d$, on the
quotient bundle over $M$. For $\xi=\xi_{1}\, dx_{1}+ \xi_{2} \, dx_{2} \in
V^{*}$, one can define a flat $U(1)$ connection on the complex line bundle $%
\underline{\mathbb{C}\, }_{\mid_{M}} : M \times \mathbb{C} \rightarrow M$ by 
\begin{equation}  \label{conn1}
\nabla^{\xi} := d + i \xi.
\end{equation}

It is a simple fact that the gauge equivalence
classes of flat line bundles on $M$ are parametrized by $M^{*} := V^{*} / \,
2 \pi \Lambda^{*}$. We write 
\begin{equation}
\overline{L}_{[\xi]} := \big(\, \underline{\mathbb{C}\, }_{\mid_{V}} /
\Lambda, \nabla^{\xi} \, \big)
\end{equation}
for the flat line bundle on $M$ corresponding to the connection $%
\nabla^{\xi} $, $\xi \in V^{*}$. With the connection $\nabla^{\xi}$, it is
seen that the parallel transport along the loops is given by $\chi_{\xi}$. 

Remark that in $(\ref{conn1})$ the sign convention is
actually consistent with that in \cite{DK} as far as $\overline{L}_{[\xi]}$
is concerned, because by \cite[proof of Proposition 2.2.3]{DK} as remarked
in \cite[p. 83]{DK}, their $L_{\xi}$ is seen to be the same as $\overline{L}%
_{[\xi]}$ above; see also \cite[proof of Lemma 3.2.14, p. 86]{DK}. 

Dually, for any given $x \in V$ we define a character $%
\chi_{x} : 2 \pi \Lambda^{*} \rightarrow U(1)$ by 
\begin{equation}
\chi_{x}(2 \pi \nu) := e^{ -2 \pi i \, <\nu, x>}.
\end{equation}
So we get flat line bundles $\overline{L}_{[x]}$ over $M^{*}$ with parallel
transport $\chi_{x}$. 

The above picture paves the way for the following lemma. 

\begin{lemma}
\label{Poincare2} There is a complex line bundle $\mathcal{P}
$ over $M \times M^{*}$ with a unitary connection, such that the restriction
of $\mathcal{P}$ to each $M_{[\xi]} :=M \times \{ [\xi] \}$ is isomorphic
(as a line bundle with connection) to $\overline{L}_{[\xi]}$ and the
restriction to each ${M^{*}}_{[x]} : = \{ [x] \}\times M^{*}$ is isomorphic
to $\overline{L}_{[x]}$. 
\end{lemma}

To be more precise, we consider the connection 1-form $%
\mathbb{A}= i \xi$, $\xi \in V^{*}$ on the trivial line bundle $\underline{%
\mathbb{C}\, }_{\mid_{M \times V^{*}}} : \big( M \times V^{*} \big) \times 
\mathbb{C} \rightarrow M \times V^{*}$. We can lift the actions of $2 \pi
\Lambda^{*}$ on $M \times V^{*}$ to $\underline{\mathbb{C}\, }_{\mid_{M
\times V^{*}}}$ by 
\begin{equation}  \label{action2pilambda}
2 \pi \nu \circ (x, \xi, \sigma) := (x,\, \xi + 2 \pi \nu,\, e^{- 2 \pi i
<\nu, x>}\ \sigma ), \hspace*{20pt} \forall \nu \in \Lambda^{*}.
\end{equation}
This action preserves the connection $d+\mathbb{A}$ and hence induces a
connection on the line bundle 
\begin{equation}  \label{bbP}
\mathcal{P} :=\underline{\mathbb{C}\, }_{\mid_{M \times V^{*}}} / \, 2 \pi
\Lambda^{*} \rightarrow M \times M^{*},
\end{equation}
\noindent denoted as $\nabla^{\mathcal{P}}$. It is worthwhile mentioning
that although the connection is flat on each slice $\mathcal{P}_{\mid_{M
\times \{ [\xi]\}}} \cong \overline{L}_{[\xi]}$, it is \textbf{not} flat on
the entire $\mathcal{P}$. Indeed the curvature of $\nabla^{\mathcal{P}}= d+ 
\mathbb{A}$ is 
\begin{equation}  \label{dA}
d \mathbb{A} + \mathbb{A} \wedge \mathbb{A} = i\, ( d \xi_{1}\wedge dx_{1} +
d \xi_{2}\wedge dx_{2} ).
\end{equation}

\noindent Similarly, if we define a metric $h_{\mathbb{%
\underline{\mathbb{C}\, }}_{\mid_{M \times V^{*}}}}(x, \xi) \equiv 1$ on the
trivial line bundle $\underline{\mathbb{C}\, }_{\mid_{M \times V^{*}}}$, or
equivalently, 
\begin{equation}  \label{good metric}
<\, (x, \xi, \sigma_{1}), (x, \xi, \sigma_{2}) \,>_{\underline{\mathbb{C}\, }%
_{\mid_{M \times V^{*}}}}\ := \ \sigma_{1} \cdot \overline{\sigma_{2}},
\end{equation}
\noindent then the metric $(\ref{good metric})$ is preserved by the action
of $2 \pi \Lambda^{*}$ in $(\ref{bbP})$. Thus it induces a metric on $%
\mathcal{P}$, denoted as $h_{\mathcal{P}}$. 

One sees that the connection $\nabla^{\mathcal{P}}$ and the
metric $h_{\mathcal{P}}$ just defined are compatible on $\mathcal{P}$, that
is, the connection is unitary with respect to the metric as required in
Lemma \ref{Poincare2}. 

The holomorphic structure on the line bundle $\mathcal{P}$
is discussed in the next section. 

\section{Identify $\mathcal{P}$ with the Poincar\'e line
bundle $\mathrm{P}$}

\label{section 6} 

The following lemma is almost immediate. It is included to
make the transformation in coordinates more transparent. 

\begin{lemma}
\label{MM} One has 
\begin{equation*}
Iso : \widehat{M} \xrightarrow{\  \ \sim \  \ } M^{*}.
\end{equation*}
\end{lemma}

\begin{proof}
Recall that $\widehat{M}={\rm Pic}^{0}(M)  \cong   H_{\overline{\partial}}^{0,1}(M) / \, H^{1}(M, \mathbb{Z})$ with the image of $H^{1}(M, \mathbb{Z})$ in $H_{\overline{\partial}}^{0,1}(M)$ as $  \overline{\Lambda}^{\, *} =\{ \, n_{1}\, dx_{1}^{*} + n_{2}\, dx_{2}^{*} \mid n_{1}, n_{2} \in \mathbb{Z} \, \}$ in the notations of Section \ref{section 2}. We  write 

\begin{align}
{\rm Pic}^{0}(M) = \widehat{M} &= \frac{\{ \, c_{1}\, dx_{1}^{*} + c_{2}\, dx_{2}^{*} \mid c_{1}, c_{2} \in \mathbb{R} \, \} }{ \{ \, n_{1}\, dx_{1}^{*} + n_{2}\, dx_{2}^{*} \mid n_{1}, n_{2} \in \mathbb{Z} \, \}}  \\
 &= \frac{\{ \, \hat{\mu}\, e_{1}^{*}  \mid \hat{\mu} \in \mathbb{C} \, \} }{ \{ \, (m_{1} + m_{2}\, \frac{\tau}{\delta})\
 e_{1}^{*} \mid m_{1}, m_{2} \in \mathbb{Z} \, \}} \nonumber
\end{align}

\noindent where $\hat{\mu}= \hat{\mu_{1}} + i \hat{\mu_{2}}$. Similarly, from $(\ref{M*})$,
\begin{equation}
M^{*} = \frac{\{ \, \xi_{1}\, dx_{1} + \xi_{2}\, dx_{2} \mid \xi_{1}, \xi_{2} \in \mathbb{R} \, \} }{ \{ \, 2 \pi \, k_{1}\, dx_{1} + 2 \pi \, k_{2}\, dx_{2} \mid k_{1}, k_{2} \in \mathbb{Z} \, \}}.
\end{equation}
\noindent We have the group isomorphism $Iso : \widehat{M} \rightarrow M^{*}$  by sending $dx_{1}^{*}$ to  $ 2\pi \, dx_{1}$ and $dx_{2}^{*}$ to  $ 2\pi \, dx_{2}$ with
\begin{equation} \label{Iso} \hspace*{-8pt}
\begin{cases} \xi_{1} = \frac{ 2 \pi \delta}{\tau_{2}}\, \hat{\mu}_{2} \\
\xi_{2}= \frac{2 \pi}{\tau_{2}}\, (\tau_{1} \hat{\mu}_{2}- \tau_{2}\hat{\mu}_{1}) \end{cases}
\hspace*{-8pt} \mbox{equivalently $\begin{cases} \hat{\mu}_{1} = \frac{-1}{2 \pi}\, \xi_{2} + \frac{1}{2 \pi} \frac{\tau_{1}}{\delta}\, \xi_{1} \\
\hat{\mu}_{2}= \frac{1}{2 \pi} \frac{\tau_{2}}{\delta}\, \xi_{1} \end{cases}$.
 }
\end{equation}

\noindent In particular, $Iso\, (\overline{\Lambda}^{\, *})=2\pi \Lambda^{*}$. 
\end{proof}

Recall the line bundle $\mathrm{P} \rightarrow M \times \widehat{M}$
of Lemma \ref{Poincare1}. By the above lemma, $M^{*}$ admits a complex
structure inherited from that of $\hat M$. To compare $\mathrm{P}$ and $\mathcal{P}$,
we first note that the global connection $\nabla^{\mathcal{P}}$ in the
preceding section on the line bundle $\mathcal{P} \rightarrow M \times M^{*} 
$ of Lemma \ref{Poincare2} gives a holomorphic structure on $\mathcal{P}$
(where the $M$ has been identified with the previous $M$ automatically as a
complex torus). 

To see this, define 
\begin{equation*}
\widetilde{Iso} \ : = \left(\, Id,\ Iso \, \right): M\times \widehat{M}
\rightarrow M\times M^{*}
\end{equation*}
with $Iso : \widehat{M} \rightarrow M^{*}$ in Lemma \ref{MM}. Let's form the
pull-back bundle $\widetilde{Iso}^*\mathcal{P}$ equipped with the pull-back
metric $\widetilde{Iso}^{*} h_{\mathcal{P}}$ and the pull-back connection $%
\tilde\nabla:=\widetilde{Iso}^{*} \nabla^{\mathcal{P}}$. By $\nabla^{%
\mathcal{P}} = d+ i \xi$, the connection is seen to be 
\begin{equation*}
\tilde\nabla=d + \frac{\pi}{\tau_{2}} ( - \overline{\hat{\mu}}\, dz + \hat{%
\mu}\, d \overline{z})
\end{equation*}
and the curvature $\Theta_{\tilde\nabla}$ of $\tilde\nabla$ is 
\begin{equation}  \label{curvpullback}
d \mathbb{A} + \mathbb{A} \wedge \mathbb{A} = \frac{\pi}{\tau_{2}}(d z
\wedge d \overline{\hat{\mu}} + d \hat{\mu} \wedge d \overline{z}).
\end{equation}

\noindent Remark that the calculation to derive $(\ref{curvpullback})$
is merely to plug $(\ref{Iso})$ and $(\ref{x-zcoord})$ into $(\ref{dA})$.
Now that the curvature of $\tilde\nabla$ is of type $(1,1)$, it is
well-known that $\tilde\nabla$ gives rise to a holomorphic structure on $%
\widetilde{Iso}^*\mathcal{P}$. This implies the above claim. 

We shall now identify $\mathrm{P}$ and $\mathcal{P}$. 

\begin{theorem}
\label{P2P} In the notations as above, let $\mathrm{P} \rightarrow M
\times \widehat{M}$ be the Poincar\'e line bundle of Lemma \ref{Poincare1},
and $\mathcal{P} \rightarrow M \times M^{*} $ of Lemma \ref{Poincare2} be
equipped with the holomorphic structure as given precedingly. Then 
\begin{equation}
\mathrm{P} \cong \widetilde{Iso}^*\mathcal{P}.
\end{equation}
\end{theorem}

\begin{proof}
  By Lemma \ref{Poincare1} , $\rm P$ is the unique holomorphic line bundle on $M \times \widehat{M}$ satisfying \vspace{-4pt}\\

\noindent $(1)\   {\rm P}_{\mid_{M \times \{ \hat{\mu} \} } } \cong {\rm P}_{\hat{\mu} }$.\vspace*{-4pt} \\
$(2)\  {\rm P}_{\mid_{\{0\} \times \widehat{M}}}$ is holomorphically trivial on $\{ 0 \} \times \widehat{M}$.\\

To show that ${\rm P} \cong   \widetilde{Iso}^*\mathcal{P}$ where
$\widetilde{Iso}= \left(\,  Id,\ Iso \, \right)$ as defined prior to Theorem \ref{P2P},
it therefore suffices to prove the following for $\mathcal{P}\rightarrow M \times M^{*}$: \\

\noindent $(1')$\ for any $[\xi] \in M^{*}$, the line bundle $\overline{L}_{[\xi]} \cong \mathcal{P}_{\mid_{M \times \{ [\xi] \}}}$ is holomorphically isomorphic to ${\rm P}_{Iso^{-1} \, ([\xi])}= {\rm P}_{\hat{\mu}}$.\\
\noindent $(2')$\ $\mathcal{P}_{\mid_{\{0\} \times M^{*}}}$ is holomorphically trivial on $\{0\}\times M^{*}$.\\

To prove $(1')$,
\noindent from the action in $(\ref{actionlambda})$ that
$$\lambda \circ (x, \sigma) = (x+\lambda,\, \chi_{\xi}(\lambda)\, \sigma)= (x+\lambda,\, e^{-i<\xi, \lambda>}\, \sigma),$$
\noindent the holonomy transforms the basis $\lambda$ by $\chi_{\xi}(\lambda)$ as remarked earlier. Accordingly,
the multipliers of $\overline{L}_{[\xi]}$ which transforms inversely, are
\begin{equation}\label{eqnLbar}
\begin{cases}  e_{\lambda_{1}}(z) = e^{\, i \xi_{1}} \\ e_{\lambda_{2}}(z) = e^{\, i \xi_{2}}. \end{cases} \hspace*{210pt}
\end{equation}
\noindent Recall that the multipliers of ${\rm P}_{\hat{\mu}}$ are (cf. $(\ref{Pmu})$, $(\ref{eqn21})$ and
the complex linearity of $(\ref{eqn 17})$)
\begin{equation}\label{eqn57}
\begin{cases}  e_{\lambda_{1}}(z) = 1 \\ e_{\lambda_{2}}(z) = e^{-2 \pi i \hat{\mu}}. \end{cases} \hspace*{200pt}
\end{equation}

  To match the above two sets of multipliers $(\ref{eqnLbar})$ and 
$(\ref{eqn57})$, define a line bundle $L_{\Delta, \xi} \rightarrow M$ with the (constant) multipliers
\begin{equation}\label{eqn58}
\begin{cases}  e_{\lambda_{1}}(z) = e^{\, ia\delta}=e^{\,  i\xi_{1}} \\
e_{\lambda_{2}}(z) = e^{\, ia\tau}= e^{\, i \frac{\tau}{\delta} \xi_{1}} \end{cases} \hspace*{160pt}
\end{equation}
\noindent where $a=\frac{\xi_{1}}{\delta} \in \mathbb{R}$. The function
\begin{equation}
\Phi_{\xi}(z)=e^{\, i a z} 
\end{equation}
\noindent satisfying the quasi-periodic property with respect to $(\ref{eqn58})$ (see Section \ref{section 1-2} and (\ref{2.0})) 
is then a global, nowhere vanishing section of $L_{\Delta, \xi}$.
Therefore $L_{\Delta, \xi}$ is holomorphically trivial on $M$.

  Via $(\ref{eqn57})$ and $(\ref{eqn58})$, the multipliers of the line bundle $\rm{P}_{\hat{\mu}} \otimes L_{\Delta, \xi}$ become
\begin{equation}
\begin{cases}  e_{\lambda_{1}}(z)= 1 \cdot  e^{\,  i\xi_{1}} = e^{\,  i\xi_{1}} \\
e_{\lambda_{2}}(z) = e^{-2 \pi i \hat{\mu}} \cdot e^{\, i \frac{\tau}{\delta} \xi_{1}} = e^{\, i \xi_{2}}, \end{cases} \hspace*{140pt}
\end{equation}
where the second multiplier uses $(\ref{Iso})$.
\noindent Therefore, $\overline{L}_{[\xi]} \cong {\rm P}_{\hat{\mu}}$ holomorphically, proving $(1')$.

It remains to prove $(2')$.
\noindent Recall that the action in $(\ref{action2pilambda})$
$$2 \pi \nu \circ (x, \xi, \sigma) := (x,\, \xi + 2 \pi \nu,\, e^{-2 \pi i <\nu, x>}\  \sigma ), \hspace*{20pt} \forall \nu \in \Lambda^{*}.$$
\noindent At $x=0$, this becomes
\begin{equation}
2 \pi \nu \circ (0, \xi, \sigma) = (0,\, \xi + 2 \pi \nu,\, \sigma ) \hspace*{20pt} \forall \nu \in \Lambda^{*}.
\end{equation}
\noindent Since $\sigma$ is unchanged, it follows that
$\mathcal{P}_{\mid_{\{0\} \times M^{*}}}$ has trivial multipliers and hence a holomorphically trivial line bundle on $M^{*}$,
proving $(2')$.
\end{proof}

\section{Main Results}

\label{section 7} 

We shall now organize our preceding results and prove our
main results here. By Theorem \ref{P2P} that $\rm{P}\cong \widetilde{Iso}^{*}%
\mathcal{P}$, we can pull back the metric $h_{\mathcal{P}}$ and the
connection $\nabla^{\mathcal{P}} = d+ i \xi$ on $\mathcal{P}$ via the map $%
\widetilde{Iso}$, and get a metric and a compatible connection on $\rm{P}$

 \begin{equation} 
 h_{\rm{P}} := \widetilde{Iso}^{*}h_{\mathcal{P}}, \hspace*{10pt}
\nabla^{\rm{P}} := \widetilde{Iso}^{*} \nabla^{\mathcal{P}}. 
\end{equation} 

\noindent Write $\Theta_{\rm{P}} $ for the curvature of $\nabla^{\rm{P}}$.
 If we combine $(\ref{curvpullback})$ with Theorem $\ref{theorem 2}$ in
Section $\ref{section 3}$ (see also ($\ref{eqn 17}$)), we have the first part of the following theorem.  

\begin{theorem} \label{c=1} Recalling that $h_{(Id \times
\varphi_{L_{0}})^{*}\rm{P}}$ and $\Theta_{(Id \times \varphi_{L_{0}})^{*}\rm{P}}$ on $%
(Id \times \varphi_{L_{0}})^{*}\rm{P}\to M\times M$ (see $(\ref{metric IdPhiP})$
and $(\ref{curvP})$), one has the following. On $M\times M$, \newline
$(1)\ (Id \times \varphi_{L_{0}})^{*}\Theta_{\rm{P}} = \Theta_{(Id \times
\varphi_{L_{0}})^{*}\rm{P}}$. \hspace*{10pt} \newline
$(2) \ (Id \times \varphi_{L_{0}})^{*} h_{\rm{P}} =h_{(Id \times
\varphi_{L_{0}})^{*}\rm{P}}$. \hspace*{10pt} \newline 
\end{theorem} 

\begin{proof} The first part of the theorem
is just noted. In turn, it yields that the two metrics in the second
statement differ at most by a multiplicative constant $c$. If one restricts
both metrics to $\{0\} \times M$, one sees that $c=1$. 
\end{proof} 
 
To proceed further, we form some vector bundles as follows. %

\begin{definition} Define the line bundles

\begin{align} 
\widetilde{E} &:= \pi_{1}^{*}L_{0} \otimes \rm{P} \rightarrow M \times \widehat{M}%
. \\ \widetilde{E^{\prime }}&:= \pi_{1}^{*}L_{0} \otimes (Id \times
\varphi_{L_{0}})^{*} \rm{P} \rightarrow M_{1} \times M_{2} \hspace*{50pt} 
\end{align}

\noindent where $M_{1} \cong M_{2} \cong M $, and the vector bundles 
\begin{equation*} {E} := (\pi_{2})_{*} \widetilde{E} \rightarrow \widehat{M}%
, \hspace*{10pt} {E^{\prime }} := (\pi_{2})_{*} \widetilde{E}\rightarrow
M_{2}. \hspace*{40pt} 
\end{equation*} 
\end{definition}

The transformation from $L_0\to M$ to $E\to \hat M$ (or $E^{\prime }\to M$)
can be placed in the context of the so-called Fourier-Mukai transform, but
we shall not go into it here. We refer to \cite[Section 5]{P2} for more
details.  

In what follows, we shall interchangeably use
the identification $\rm{P}\cong \widetilde{Iso}^{*}\mathcal{P}$ obtained in
Theorem \ref{P2P}. First equip $\widetilde{E}$, $\widetilde{E^{\prime }}$
with metrics 

\begin{equation} \label{hE} h_{\widetilde{E}}= \pi_{1}^{*}
h_{L_{0}} \otimes h_{\rm{P}} \mbox{\hspace*{10pt} where \ $ h_{\rm P}= \widetilde{%
Iso}^{*} h_{\mathcal{P}} $, } 
\end{equation} 

\noindent (cf. $(\ref{hLo})$
for $h_{L_0}$ and $(\ref{metric IdPhiP})$) 

\begin{equation} \label{hE'} h_{%
\widetilde{E^{\prime }}}= \pi_{1}^{*} h_{L_{0}} \otimes h_{(Id \times
\varphi_{L_{0}})^{*} \rm{P}} 
\end{equation} 

\noindent respectively. By $(2)$ of Theorem 
\ref{c=1}, one has 

\begin{equation} \label{hEhE'} h_{\widetilde{E^{\prime }}%
}= (Id \times \varphi_{L_{0}})^{*}h_{\widetilde{E}}. 
\end{equation} 
 
We shall now equip the vector bundle ${E}$ with a metric
given by the $L^{2}$-metric on ${E}_{\mid_{\hat{\mu}}} = H^{0}(M, L_{\hat{\mu%
}})$ using $h_{\widetilde{E}}$, and similarly the $L^{2}$-metric on $%
{E^{\prime }}_{\mid_{\mu}} = H^{0}(M, L_{\mu})$ using $h_{\widetilde{%
E^{\prime }}}$. These $L^2$-metrics on ${E}$ and ${E^{\prime }}$ are denoted
by $h_{{E}}$ and $h_{{E^{\prime }}}$ respectively.

Recall that $\widetilde{K} = \pi_{1}^{*}L_{0} \otimes (Id \times
\varphi_{L_{0}})^{*}\rm{P} \otimes \pi_{2}^{*}L_{0}$. By the explicit expressions 
$(\ref{K metric formula})$ and $(\ref{hLo})$, one sees that 

\begin{equation}
h_{\widetilde{E^{\prime }}} = e^{\frac{- 2 \pi}{\tau_{2}}
({z_{2}}^{2}+2z_{2}\mu_{2})}. 
\end{equation} 

We summarize the above in the following. 

\begin{proposition} %
\label{prop2} \begin{equation} (\widetilde{E^{\prime }}, h_{\widetilde{%
E^{\prime }}}) = (Id \times \varphi_{L_{0}})^{*}(\widetilde{E}, h_{%
\widetilde{E}}) \end{equation} where $h_{\widetilde{E}}$ and $h_{\widetilde{%
E^{\prime }}}$ are defined as in $(\ref{hE})$ and $(\ref{hE'})$. As a
consequence, \begin{equation} ({E^{\prime }}, h_{{E^{\prime }}}) =
\varphi_{L_{0}}^{*}({E}, h_{{E}}) \end{equation} with the curvatures \begin{%
equation} \Theta ({E^{\prime }}, h_{{E^{\prime }}}) =
\varphi_{L_{0}}^{*}\Theta({E}, h_{{E}}). 
\end{equation}  
\end{proposition} 

Recall that ${K} \rightarrow M_{2}$ is
the vector bundle ${K}_{\mid_{\mu}} = H^{0}(M, \widetilde{K}_{\mid_{M\times
\{\mu \}}})$ of Section $\ref{section 4}$. As vector bundles 

\begin{equation%
} {K} = {E^{\prime }} \otimes L_{0}, \hspace*{4pt} {E^{\prime }} = {K}
\otimes L_{0}^{*} \mbox{\hspace*{4pt} where $L_{0}^{*}$ is the dual of $L_{0}
$}. 
\end{equation} 

\noindent By Theorem \ref{KK} that ${K}$ splits into line bundles (each of which is holomorphically trivial)

\begin{equation} \label{Esplit} {K} = \overset{\delta -1}{%
\underset{m=0}{\bigoplus}} {K}_{m},
\end{equation}

\noindent it follows that 

\begin{equation} \label{Esplit1} {E^{\prime }} = {K}
\otimes L_{0}^{*}= \overset{\delta -1}{\underset{m=0}{\bigoplus}}
\big({K}_{m}\otimes L_{0}^{*} \big)= \overset{\delta -1}{\underset{m=0}{\bigoplus}}
L_{0}^{*}. 
\end{equation}  

By Theorem \ref{KK}, $(\ref{Esplit1})$, and $(\ref{curv2})$, the curvature of $%
{E^{\prime }}$ is immediately computed as follows. 

\begin{theorem} 
\label{theorem 7} Let's denote by $\big( Id \big)_{\delta
\times \delta}$ the $\delta \times \delta$ identity matrix. Then we have
 
\begin{equation} \Theta ({E^{\prime }}, h_{{E^{\prime }}}) = -
\Theta_{L_{0}}(\mu) \big( Id \big)_{\delta \times \delta} = \frac{- \pi}{%
\tau_{2}}\, d \mu \wedge d \overline{\mu}\ \big( Id \big)_{\delta \times
\delta}. 
\end{equation} 
\end{theorem}

Combining Theorem $%
\ref{theorem 7}$ and Proposition $\ref{prop2}$ (see also ($\ref{eqn 17}$)), we have 

\begin{theorem} 
\label{theorem 8}  \noindent $(1)\
\Theta({E}, h_{{E}}) = - \frac{\pi}{\tau_{2}}\, d\hat{\mu} \wedge d 
\overline{\hat{\mu}}\ \big( Id \big)_{\delta \times \delta}$.\newline
\noindent $(2)$\ As a consequence of $(1)$, the first Chern class of ${E}$
is 

\begin{equation} c_{1}({E}, h_{{E}}) = \frac{-i \delta}{2 \tau_{2}}\, d 
\hat{\mu} \wedge d \overline{\hat{\mu}} 
\end{equation} (at the level of
differential forms).  
\end{theorem} 

\begin{remark} \label{rem85} i) Our computational result of $c_1(E)$ agrees
with that of the torus case in \cite[Theorem 12]{P2} of Prieto, in view of
his Remark 10 and various notations in p. 388, p. 381 and p. 386. 

ii) It is unclear to us whether Theorem \ref{theorem 8} can be
proved independently of Theorem \ref{theorem 7}, mainly due to the fact that
our description of ($\mu$-dependent) theta functions is most conveniently
given on $M\times M$ rather than on $M\times \hat M$, as remarked earlier in
Introduction. 
\end{remark}

\end{document}